\documentclass[12points,reqno]{amsart}
\usepackage{amsfonts}
\usepackage{amssymb}
\usepackage{graphicx}
\usepackage{amsmath}

\newtheorem{theorem}{Theorem}
\newtheorem{lemma}{Lemma}
\newtheorem{prop}{Proposition}
\newtheorem{corollary}{Corollary}
\theoremstyle{definition}

\theoremstyle{remark}
\newtheorem{remark}{Remark}

\numberwithin{equation}{section}

\everymath{\displaystyle}

\begin{document}
\title{ Finite generation of Lie derived powers of associative algebras }

%

\author{Adel Alahmadi}
\address{Department of Mathematics, Faculty of Science, King Abdulaziz University, P.O.Box 80203, Jeddah, 21589, Saudi Arabia}
\email{analahmadi@kau.edu.sa}
\author{Hamed Alsulami}
\address{Department of Mathematics, Faculty of Science, King Abdulaziz University, P.O.Box 80203, Jeddah, 21589, Saudi Arabia}
\email{hhaalsalmi@kau.edu.sa}
%
%

\maketitle

\begin{abstract}
Let $A$ be an associative algebra over a field of characteristic $\neq 2$ that is generated by a finite collection of nilpotent elements. We prove that all Lie derived powers of $A$ are finitely generated Lie algebras.
\end{abstract}

\maketitle

\section{Introduction and  Main Result}

All algebras are considered over a field $F$ of  characteristic $\neq 2.$  \newline An associative algebra $A$ gives rise to the Lie algebra $A^{(-)}=(A, [a,b]=ab-ba).$  In this note we address the question first raised by I. Herstein [4] (see also [1]):

given a finitely generated associative algebra $A$ when is the Lie algebra $[A,A]$ finitely generated?

Consider the derived series of the Lie algebra $$A^{(-)}: A^{(-)}=A^{[0]}> A^{[1]}>\cdots, A^{[i+1]}=[A^{[i]},A^{[i]}].$$

Recall that an element $a\in A$ is called \underline{nilpotent} if there exists an integer $n(a)\geq 1$ such that $a^{n(a)}=0.$ We say that $A$ is a\underline{ nil algebra} if every element of $A$ is nilpotent.

E. S. Golod [3] and T. Lenagan, A. Smoktunowicz [6]  constructed examples of infinite dimensional finitely generated nil algebras.

\medskip

\begin{theorem}
Let $A$ be an associative algebra generated by a finite collection of nilpotent elements. Then an arbitrary derived power $A^{[i]}, i\geq1,$ is a finitely generated Lie algebra.
 \end{theorem}

 \begin{corollary}
 Let $A$ be a finitely generated nil algebra ( see [3], [6] ). Then the Lie algebras $A^{[i]}, i\geq 1,$ are finitely generated.
\end{corollary}
 The key sufficient condition for finite generation of $A^{[i]}$ ( see the Proposition 1 below) is essentially based on the work  of C. Pendergrass ( see [7] ). In particular, the following
lemma is a version of Lemma 3 in [7].

\begin{lemma}
 Let $U$ be an ideal of the Lie algebra $A^{(-)}.$ Then $[id_A([U,U]),A]\subseteq U$
\end{lemma}

 \begin{proof}
Since $[U,U]$ is also an ideal of $A^{(-)}$ it follows that $id_A([U,U])=[U,U]+[U,U]A.$ Denote $x\circ y =xy+yx.$ The following identities are well known:
\begin{enumerate}
  \item $xy=\frac{1}{2}([x,y]+x\circ y)$
  \item $[z,x\circ y]=[z,x]\circ y+[z,y]\circ x$
  \item  $[z,x\circ y]=[z\circ x, y]+[z\circ y, x]$
\end{enumerate}

By the identity $(1)$  $$[U,U] A\subseteq [[U,U],A]+[U,U]\circ A.$$ Clearly $[[U,U],A]\subseteq U$ and $[[[U,U],A],A]\subseteq U.$

By the identity $(2)$ $$[U,U]\circ A\subseteq [U,U\circ A]+[U,A]\circ U\subseteq U+ U\circ U.$$

By the identity $(3)$ $$ [U\circ U, A]\subseteq [U,U\circ A]\subseteq U,$$ which completes the proof of the lemma.
 \end{proof}

We say that an algebra $A$ is finitely graded if $A=A_1+A_2+\cdots$ is a direct sum, $ A_i A_j\subseteq A_{i+j},$ the algebra $A$ is generated by $A_1$ and $\dim_FA_1<\infty.$

\begin{prop}
 Let $A$ be a finitely graded algebra. Suppose that all factor algebras $A/id_A(A^{[i]})$ are nilpotent. Then all derived powers $A^{[i]}, i\geq 1,$ are finitely generated Lie algebras.
\end{prop}

\begin{proof}
 Let $i\geq 1.$ Consider the ideal $U=A^{[i+1]}=[A^{[i]},A^{[i]}]$ of the Lie algebra $A^{(-)}.$ By Lemma 1 we have $$[id_A(A^{[i+2]}),A]\subseteq [A^{[i]},A^{[i]}].$$
 By the assumption of the proposition there exists $n\geq1$ such that $A^n\subseteq A^{[i+2]},$ and, therefore $\sum\limits_{j\geq n} A_j\subseteq id_A(A^{[i+2]}).$
 We claim that the Lie algebra $A^{[i]}$ is generated by the finite dimensional subspace $A^{[i]}\cap \left(\sum\limits_{k=1}^{2n-2}A_k\right).$
 Indeed, since $A^{[i]}$ is a graded subspace of $A$ we need to check that if $m\geq 2n-1$ then $A^{[i]}\cap A_m$ lies in the Lie subalgebra generated by
  $A^{[i]}\cap \left(\sum\limits_{k=1}^{m-1}A_k\right).$ We have $A^{[i]}\cap A_m\subseteq \sum\limits_{p+q=m} [A_p,A_q].$
  Since $m\geq2n-1$ it follows that $p\geq n$ or $q\geq n.$ If $p\geq n$ then $A_p\subseteq id_A(A^{[i+2]})$ and therefore $$[A_p,A_q]\subseteq [id_A(A^{[i+2]},A]\subseteq [A^{[i]},A^{[i]}].$$ It implies $$ A^{[i]}\cap A_m\subseteq\sum [A^{[i]}\cap \sum\limits_{k=1}^{m-1}A_k,
A^{[i]}\cap \sum\limits_{k=1}^{m-1}A_k]$$ and completes the proof of the proposition.

 \end{proof}

\begin{proof}[Proof of Theorem 1.]
Let the algebra $A$ be generated by elements $a_1,a_2,\cdots,a_m$ such that $a_i^{n_i}=0, n_i\geq 1,\, 1\leq i \leq m.$
Consider the algebra $\tilde{A}$ presented by generators $x_1,x_2,\cdots,x_m$ and relations $x_i^{n_i}=0, n_i\geq 1,\, 1\leq i \leq m.$
The algebra $\tilde{A}$ is finitely graded and the algebra $A$ is a homomorphic image of the algebra $\tilde{A}.$ It is sufficient to prove that for an arbitrary $k\geq 1$ the Lie algebra $\tilde{A}^{[k]}$ is finitely generated. Moreover, by Proposition 1 it is sufficient to prove that all factor algebras $B_k=\tilde{A}/id_{\tilde{A}}( \tilde{A}^{[k]})$ are nilpotent.  Consider the following elements of the free associative algebra

\begin{alignat*}{2}
     f_1(x_1,x_2)&=[x_1,x_2] \\
      f_s(x_1,\cdots,x_{2^s})&=[f_{s-1}(x_1,\cdots,x_{2^{s-1}}),f_{s-1}(x_{2^{s-1}+1},\cdots,x_{2^{s}})], \text{ for } s\geq 2.
\end{alignat*}

The algebra $B_k$ is a finitely generated algebra that satisfies the polynomial identity $f_k(x_1,\cdots,x_{2^k})=0.$
Hence  the Jacobson radical $J(B_k)$ is nilpotent (see [2]). The semisimple algebra $B_k/J(B_k)$ is a subdirect product of primitive algebras satisfying the identity $f_k=0.$

It follows from Kaplansky theorem ( see [5] ) that a primitive algebra satisfying the identity $f_k=0$ is a field. Hence the algebra $B_k/J(B_k)$ is commutative.
The commutative algebra $B_k/J(B_k)$ is generated by a finite collection of nilpotent elements. Hence the algebra $B_k/J(B_k)$ is nilpotent, hence the algebra $B_k=J(B_k)$ is nilpotent. This completes the proof of the theorem.

\end{proof}

\begin{remark}
In [1] we showed that if $A$ is a finitely generated algebra with an idempotent $e$ such that $$AeA=A(1-e)A=A\,\,\,(*)$$ then the Lie algebra $A^{[1]}=[A,A]$ is finitely generated . The condition $(*)$ is equivalent to $A$ being generated by $eA(1-e)+(1-e)Ae.$ Since elements from $eA(1-e)$ and $(1-e)Ae$ are nilpotent this result follows from Theorem 1.

\end{remark}

\end{document}